\documentclass{article}
\usepackage[utf8]{inputenc}
\usepackage{amsmath}
\usepackage{graphicx}
\usepackage{multirow}
\usepackage[dvipsnames]{xcolor}
\usepackage{amssymb}
\usepackage{float}
\usepackage{tikz}
\usepackage{amsthm}
\usepackage{mathtools}
\usepackage{hyperref}
\usepackage[parfill]{parskip}
\makeatletter
\def\thm@space@setup{%
  \thm@preskip=\parskip \thm@postskip=0pt
}
\makeatother
\usepackage{enumerate}
\usepackage{etaremune}
\usepackage{comment}
\usepackage{ytableau}
\ytableausetup{aligntableaux = bottom}

\newtheorem{lemma}{Lemma}[section]
\newtheorem{defn}[lemma]{Definition}
\newtheorem{cor}[lemma]{Corollary}
\newtheorem{prop}[lemma]{Proposition}
\newtheorem*{eg}{Example}

\newtheorem*{lemmaunnumbered}{Lemma}

\DeclareUnicodeCharacter{0308}{}
\DeclareUnicodeCharacter{0306}{}

\title{Principal ideals in a plactic monoid always intersect}
\author{Daniel Turaev}
\date{}

\begin{document}

\maketitle

\begin{abstract}
This note presents a proof that two principal ideals in a plactic monoid always intersect. Namely, this means that the plactic monoids are both left and right reversible. To the author's knowledge, this result has not yet appeared in the literature studying this monoid. This result holds for both finite rank plactic monoids and the infinite rank plactic monoid.  
\end{abstract}

In any monoid $M$, one can define a right ideal generated by a set $A\subset M$ as the set $$AM = \{am\ |\ a \in A, m \in M\} \subseteq M,$$ and analogously a left ideal generated by $A$ as the set $$MA = \{ma\ |\ a \in A, m \in M\} \subseteq M.$$

If $A$ consists of a single element $s \in M$, then we call these left and right ideals \emph{principal}. We abuse notation and write $$sM = \{sm\ |\ m\in M\}$$ and $$Ms = \{ms\ |\ m\in M\}$$ for the principal right and principal left ideals generated by $s$ respectively.

We call a monoid \emph{right reversible} if any two left ideals intersect. That is,$$ \forall A, B\subset M:\ MA\cap MB \neq \emptyset.$$

Analogously, call a monoid \emph{left reversible} if any two right ideals intersect. This definition was first introduced by Dubreil \cite{dubreil}, and was used to study which semigroups are embeddable in groups \cite[Chapter 1.10]{MR132791}, but is also an interesting property in its own right. 

The following result is well known and immediate.
\begin{lemmaunnumbered}
        For any monoid $M$, if every pair of principal left (respectively right) ideals intersect, then every pair of left (respectively right) ideals intersect.
\end{lemmaunnumbered}

That is, we can equivalently define reversibility in terms of intersections of principal ideals, rather than general ideals. This is because every left (respectively right) ideal contains a principal left (respectively right) ideal.

The intersection of two principal ideals can be rephrased as finding a solution to one of the following simple word equations:
\begin{align}
    uX_1 = vY_1\label{rightideal}\\
    X_2u = Y_2v\label{leftideal}
\end{align}
respectively corresponding to the intersection of two principal right ideals and two principal left ideals. Here $u$ and $v$ are elements of a given monoid $M$, and we seek $X_i, Y_i \in M$ solving these equations.

We claim that in the plactic monoids, a family of monoids with origins in works of Schensted \cite{schensted61} and Knuth \cite{knuth1970permutations}, and later studied in detail by Lascoux and Sch\"utzenberger \cite{plaxique81}, both of the above equations will have a solution irrespective of the choice of $u$ and $v$. Thus, any two principal ideals in the plactic monoid intersect, and so they are both left and right reversible.

\section{The plactic monoids}

The quickest definition of the plactic monoids is the following.

\begin{defn}[The plactic monoids]
    Fix $n\in\mathbb{N}$. The plactic monoid of rank $n$, written $P_n$, is the monoid defined by the presentation $$P_n\cong\langle A\ |\ K_n\rangle,$$ where $A$ is the totally ordered alphabet $\{1,\dots,n\}$ and $K_n$ is the set of Knuth relations $$ K_n = \{xzy = zxy\ |\ x,y,z\in A,\  x\leq y <z \}\cup\{ yxz = yzx\ |\ x,y,z\in A, \ x<y\leq z\}.$$

    The plactic monoid of infinite rank is the monoid generated by $\mathbb{N}$ subject to the union of all Knuth relations $$P_\mathbb{N} = \langle \mathbb{N}\ |\ \bigcup_{n\in\mathbb{N}}K_n\rangle.$$
\end{defn}

Where it is clear, we will abuse notation and write $K$ for the Knuth relations, so called because they were first introduced in \cite{knuth1970permutations}. We will also write $A$ for the totally ordered alphabet $\{1,\dots,n\}$, for an arbitrary choice of $n\in\mathbb{N}$. 

Notice that if two words in $A^*$ are related via Knuth relations, they must have the same letters from $A$, though potentially arranged in different orders. There is therefore a well defined notion of content in the plactic monoids.
\begin{defn}[Content]
    Given an element $w$ of $P_n$, the content of $w$, written $c(w)$, is the tuple in $\mathbb{N}^n$ encoding how many times each letter of $A$ appears in $w$.\footnote{Equivalently, the content of an element is its Parikh image.} 
\end{defn}

While the above definition is concise, it gives no insight into the normal forms of the elements in $P_n$, nor how to correctly conceptualise the multiplication. The plactic monoids were originally defined as an insertion algorithm on Semistandard Young Tableau. We will give this description of them here, as it will allow us to talk about normal forms.

We follow the French conventions of Young diagrams having longer rows underneath shorter ones.

\begin{defn}
 A \emph{Semistandard Young Tableau} (henceforth simply tableau) is a Young diagram with labelled boxes, with labels satisfying the following conditions:
    \begin{itemize}
        \item Each row weakly increases left to right.
        \item Each column strongly decreases top to bottom.
    \end{itemize}
\end{defn}

\begin{eg} \begin{ytableau}
     4&5\\ 2& 4\\ 1&2&3
    \end{ytableau} is a tableau.
\end{eg}

To each tableau whose labels are in $A$, one can associate a \emph{row reading} and \emph{column reading}, both of which are words in $A^*$.

\begin{defn}[Row and Column readings]
    Let $t$ be a tableau with labels taken from $A$. Suppose $t$ has $m$ rows, labelled top to bottom as $r_1,\dots, r_m$. The labels of the boxes in each row are an increasing sequence, which can be viewed as a word in $A^*$, which we will also call $r_i$. The row reading of $t$ is then $w = r_1r_2\dots r_m \in A^*$. 

    Similarly, denote the columns of $t$ from left to right by $c_1,\dots, c_p$. Each such column corresponds to a strictly decreasing sequence $c_i\in A^*$. The column reading of  $t$ is then $w = c_1\dots c_p \in A^*$. 
\end{defn}

\begin{eg}
The tableau $t = $ \begin{ytableau}
        3& 4 \\2 & 3\\1 & 1 & 2 & 2
    \end{ytableau}  has row reading $$34231122 = 34\ 23\ 1122$$ and column reading $$32143122 = 321\ 431\ 2\ 2.$$
\end{eg}

One can show that the row and column forms of a tableau are Knuth equivalent to one another. Furthermore, each $K$-class in $\langle A\ | \ K\rangle$ contains exactly one tableau's row form, and that same tableau's column form. That is to say the following.

\begin{lemma}
    The row forms of tableaux with labels in $A$ are normal forms for the plactic monoid over $A$. Likewise, the column forms of tableaux with labels in $A$ are normal forms for the plactic monoid over $A$.   
\end{lemma}

See Chapter 5 of \cite{Loth} for a proof of this lemma, as well as a more detailed discussion of the above. 

For our purposes, the column form will prove a more natural normal form to work with. This normal form has some nice computational properties, which were developed in \cite{CainGray15}. We won't need the details from this paper, but it will be useful to refer to $P_n$ as being generated by $\mathcal{C}_n$, the set of columns (i.e. strictly decreasing words) in $A$. Note that when $A$ is finite, so is $\mathcal{C}_n$.

We now describe Schensted's original insertion algorithm from \cite{schensted61}.

\begin{defn}[Schensted's algorithm]
    Define $P:A^*\to A^*$ to be the map sending a word $w$ to the row reading of a tableau recursively as follows:

Firstly, $P(\varepsilon) = \varepsilon$. Then suppose $w = x_1\dots x_\ell\in A^*$ and $P(x_1\dots x_{\ell-1}) = r_1\dots r_m$, for some rows $r_i$ that form the row reading of a tableau. Then we have:\begin{enumerate}
    \item If $r_mx_\ell$ is a row, then we set $P(r_1\dots r_mx_\ell) = r_1\dots r_mx_\ell$
    \item If not, then we can write $r_m = r_\alpha y r_\beta$, with $y$ being the leftmost letter such that $x_\ell<y$. Such a $y$ must exist, since otherwise $r_mx_\ell$ would be a row. But then $r_\alpha x_\ell r_\beta$ will be a row. So we set $$P(r_1\dots r_mx_\ell) = P(r_1\dots r_{m-1}y) r_\alpha x_\ell r_\beta.$$ 
\end{enumerate}\end{defn}
We call the process in point (2) `bumping the letter $y$'. 

While the statement seems complex, it describes a fairly simple visual process of moving numbers around a tableau. Furthermore, one can equivalently define the plactic monoid as being the monoid generated by $A$ with multiplication defined by Schensted insertions. In this framework, it makes sense that the normal forms `should' be the row forms of tableaux. We again refer the reader to \cite{Loth}, or \cite{knuth1970permutations}, for a proof that this is equivalent to the monoid defined by Knuth relations.

\section{Involution monoids}

A monoid is called an \emph{involution monoid} if it admits an anti-automorphism that is an involution. That is, if a monoid $M$ has a map $\theta: M \to M$ such that $$\forall x, y: \theta(xy) = \theta(y)\theta(x)$$ and $$\forall x: \theta(\theta(x)) = x.$$

\begin{lemma}
    Let $M$ be an involution monoid. Then $M$ is left reversible if and only if it is right reversible. 
\end{lemma}
\begin{proof}

Suppose $M$ is left reversible, so Equation \ref{rightideal} holds for every pair $u, v \in M$. So we know that for some $X, Y \in M$ we have that $$\theta(u)X = \theta(v)Y,$$

since $\theta(u),\theta(v)$ are both elements of $M$. 

Now, applying $\theta$ we get that $$\theta(u)X = \theta(v)Y \implies \theta(X)\theta(\theta(u)) = \theta(Y)\theta(\theta(v)),$$ but since $\theta$ is an involution, this solves Equation \ref{leftideal} for the pair $u, v \in M$. 
\end{proof}

We note here that the plactic monoids of finite rank are involution monoids. This was first shown by Sch\"utzenberger in \cite{Schutzenberger_1963} and is known as the Sch\"utzenberger involution. This involution, denoted by $\theta$, is defined firstly as a function on the free monoid $\theta:A^*\to A^*$, acting by reversing a word and then sending each $k \in A$ to $n-k+1 \in A$. 

For example, given $A = \{1,2,3,4\},$ and $w =  32341,$ we will get an output $ \theta(w) = 41232$. As another example, $\theta(432123) = 234321$. 

It is straightforward to see that $\theta$ is an anti-automorphism on $A^*$. That is, for all $x, y \in A^*$ we have the identity $\theta(xy) = \theta(y)\theta(x).$ This map is also clearly self inverse. 

What's more, by the structure of the Knuth relations, $\theta$ is a well defined anti-automorphism on $P_n$ too. Thus, we will denote by $\theta$ the anti-automorphism from $P_n$ to itself acting by reversing the normal form as a word in $A^*$, then replacing the letters, then applying Schensted's algorithm to return to a normal form. 

\begin{cor}\label{cor:leftiffright}
    For any $n\in\mathbb{N}$, the plactic monoid of rank $n$ is left reversible if and only if it is right reversible.
\end{cor}

\section{Equations with equal content}

Henceforth we will endeavour to show that the plactic monoids are right reversible. We begin with a slightly modified version of Equation \ref{leftideal}, inspired by classical results about conjugate elements in the plactic monoids\footnote{Two elements of a monoid are (left) conjugate if they solve the equation $uX = Xv$ for some $X$. In the plactic monoids, two elements $u, v \in P_n$ are conjugate if and only if their contents match. See Theorem 17 in \cite{Conjugacy_result} for a proof of this.}. Given an arbitrary pair $u, v \in M$, consider the equation
\begin{align}\label{lefteqn}
    Xu = Xv.
\end{align}

In the plactic monoid, it is straightforward to see that the above equation only has solutions if $u$ and $v$ have equal content, since the Knuth relations do not change content. It is not immediately clear, however, that the converse is true. 

\begin{prop}\label{prop:lefteqnisgood}
    For any $n\in\mathbb{N}$, let $u, v \in P_n$ be elements with equal content. Then there is some $X \in P_n$ such that  $Xu = Xv$. 
\end{prop}
\begin{proof}
    We will prove the proposition constructively. In analysing Equation \ref{lefteqn}, it will be helpful to consider $P_n$ with respect to its column generators $\mathcal{C}_n$. Specifically, define for each $i\in A$ the column $f_i\in\mathcal{C}_n$  corresponding to the decreasing sequence of all the letters $n$ down to $i$:

\ytableausetup{centertableaux}
\begin{align*}
    f_i = \begin{ytableau}
        n\\ \scriptstyle n-1 \\ \vdots \\ i
    \end{ytableau}
\end{align*}
\ytableausetup{aligntableaux = bottom}

and consider $X$ of the form $X = f_1^{x_1}f_2^{x_2}\dots f_n^{x_n}$. Note that each $X$ here is the column form of a tableau, hence in normal form. 

When we multiply $X$ on the right by some $i \in A$, we get that \begin{align*}Xi &= f_1^{x_1}\dots f_i^{x_i}f_{i+1}^{x_{i+1}}i\dots f_n^{x_n}\\ &= f_1^{x_1}\dots f_i^{x_i+1}f_{i+1}^{x_{i+1}-1}\dots f_n^{x_n}\end{align*}
because $f_{i+1}i = f_i$, and Schensted's algorithm insists that $i$ bump the leftmost letter in the row that it can. 

Now, consider $u\in P_n$ where the number of instances of $i$ in $u$ is less than $x_{i+1}$ for each $i \in A\setminus\{n\}$. That is, $c_i(u)\leq x_{i+1}$ for each $i\in A\setminus\{n\}$. Notice that, in the product $Xu$, the bottom row of $X$ has at most as many letters $i+1$ as $u$ has letters $i$, so in the Schensted insertion algorithm, each $i$ from $u$ will bump an $i+1$ from the bottom row of $X$, regardless of which order the letters are inserted. So we obtain a formula for the product in this case
$$Xu = f_1^{x_1+c_1(u)}\dots f_i^{x_i-c_{i-1}(u) + c_i(u)}\dots f_n^{x_n - c_{n-1}(u) + c_n(u)}.$$

This then means that all letters that appeared in $u$ will remain in the bottom row of $Xu$. But that means that, so long as both $c_i(u)\leq x_{i+1}$ and $c_i(v)\leq x_{i+1}$ for all $i\in A\setminus\{n\}$, we only get $Xu\neq Xv$ when $c(u) \neq c(v)$. 

Namely, for any $u, v\in P_n$ with $c(u) = c(v)$, we can take some $X = f_1^{x_1}\dots f_n^{x_n}$ with $x_1 = 0$ and $x_i \geq c_{i-1}(u)=c_{i-1}(v)$ for each $i\in \{2,\dots, n\}$. Then we will be exactly in the above case, and so we must have $Xu = Xv$. \end{proof}

\section{Intersections of principal ideals}

We now have all the pieces to show that

\begin{prop}\label{mainprop}
    In a plactic monoid of any finite rank $n$, two principal ideals always intersect. 
\end{prop}

\begin{proof}
    By Proposition \ref{prop:lefteqnisgood}, Equation \ref{lefteqn} has a solution if and only if $u$ and $v$ have the same content. But given any pair of elements $u, v \in P_n$, one can find $\alpha, \beta \in P_n$ such that $$c(\alpha u) = c(\beta v).$$ 

    So we can always find some $X$ such that $X\alpha u = X \beta v$, which solves Equation \ref{leftideal} for the pair $u, v \in P_n$. Thus every pair of principal left ideals intersect. Then by Corollary \ref{cor:leftiffright}, every pair of principal right ideals will also intersect. Finally, every principal right ideal trivially intersects every principal left ideal, since we can solve the equation $uX = Yv$ with the pair $X = v,\ Y = u$.  
\end{proof}

Namely, Proposition \ref{mainprop} implies that the plactic monoids of finite rank are both left and right reversible.

We can further show the following nice corollary.
\begin{cor}
    The infinite rank plactic monoid $P_\mathbb{N}$ is both left and right reversible.
\end{cor}
\begin{proof}
    Given a pair $u, v \in P_\mathbb{N}$, since both $u$ and $v$ are finite products of generators, there is some $n\in\mathbb{N}$ such that $u, v \in P_n \subset P_\mathbb{N}$. So, considering their principal right ideals in $P_n$, by Proposition \ref{mainprop}, $uP_n$ and $vP_n$ intersect. But since $P_n\subset P_\mathbb{N}$, it follows that $uP_n\subset uP_\mathbb{N}$ and $vP_n\subset vP_\mathbb{N}$. Thus $uP_\mathbb{N}$ and $vP_\mathbb{N}$ intersect. The argument for left ideals is analogous.
\end{proof}

\section*{Acknowledgements}
I thank Robert Gray for posing me this question, and Mikhail Volkov for his suggestions on how to restructure this note for a clearer presentation style and to remove some redundant information.

\end{document}